\documentclass[12pt,oneside]{amsart}
\usepackage{xurl}

\usepackage{fullpage}
\usepackage{amssymb,amsmath,amsthm,amsfonts,mathrsfs}
\usepackage{enumitem}
\usepackage{mathtools}
\usepackage{pgfplots}
\pgfplotsset{compat=1.15}
\usetikzlibrary{arrows}
\usepackage{tensor}
\usepackage{tikz}
\usetikzlibrary{shapes.geometric, arrows,backgrounds}
\usepackage{tikz-cd}
\usepackage{xcolor}
\usepackage[hidelinks]{hyperref}
\usepackage{url}

\usepackage{graphicx}

\usepackage{etoolbox}

\setlist{topsep=3pt,itemsep=6pt}

\newcommand{\RR}{\mathbb{R}}
\newcommand{\R}{\mathbb{R}}

\newcommand{\Z}{\mathbb{Z}}

\def \bfA {\mathbf{A}}
\def \bfX {\mathbf{X}}

\DeclareMathOperator{\image}{image}

\newcommand{\cl}[1]{\mathcal{#1}}

\newcommand{\fk}[1]{\mathfrak{#1}}

\def \calO {\cl O}
\def \define {\emph}

\numberwithin{equation}{section}
\counterwithin{figure}{section}
\AtBeginEnvironment{figure}{\setcounter{figure}{\value{equation}}\stepcounter{equation}}

\newtheorem{theorem} [equation] {Theorem}
\newtheorem*{theorem*}              {Theorem}
\newtheorem{corollary}   [equation]  {Corollary}
\newtheorem*{corollary*}   {Corollary}
\newtheorem{lemma}       [equation]  {Lemma}
\newtheorem{exercise}       [equation]  {Exercise}

\newtheorem*{proposition*}   {Proposition}

\theoremstyle{definition}

\newtheorem*{definition*}  {Definition}

\theoremstyle{remark}
\newtheorem{remark}      [equation]  {Remark}

\newtheorem*{remark*}        {Remark}
\newtheorem{example}     [equation]  {Example}
\newtheorem*{example*}              {Example}

\def\eor{\unskip\ \hglue0mm\hfill$\diamond$\smallskip\goodbreak}
\def\eoe{\unskip\ \hglue0mm\hfill$\between$\smallskip\goodbreak}

\newenvironment{enumeratea}{
  \begin{enumerate}[label = (\alph*)]}{
  \end{enumerate}}

\newif\ifdebug

\debugfalse

\definecolor{jaw}{rgb}{0,.5,0}

\newcommand{\printname}[1]
{\ifmmode{ \smash{ \raisebox{5pt}{\text{\tiny{\textcolor{blue}{#1}}}} } }
 \else   
             \marginpar{ 
\smash{ \makebox[0pt]{\raisebox{-6pt}
                     {\tiny{\textcolor{blue}{#1}} } 
                                            } } 
}          
\fi}

\newcommand{\printnamesec}[1]
{\smash{ \makebox[0pt]{\raisebox{-6pt} {\tiny{\textcolor{blue}{#1}} } 
                                            } } } 

\newcommand{\labell}[1]{\ifdebug{\label{#1}\printname{#1}}
                         \else {\label{#1}} \fi}                                     
\newcommand{\labelsec}[1]{\ifdebug{\label{#1}\printnamesec{#1}}
                         \else {\label{#1}} \fi}                                     

\tikzstyle{arrow1} = [thick]
\tikzstyle{arrow2} = [thick,->,>=stealth]

\begin{document}



\title{Diffeological submanifolds and their friends}

\author{Yael Karshon}
\address{Department of Mathematics, University of Toronto, Toronto, ON, Canada,
and School of Mathematical Sciences, Tel-Aviv University, Tel-Aviv, Israel}
\email{karshon@math.toronto.edu, yaelkarshon@tauex.tau.ac.il}

\author{David Miyamoto}
\address{Department of Mathematics, University of Toronto, Toronto, ON, Canada}
\email{david.miyamoto@mail.utoronto.ca}

\author{Jordan Watts}
\address{Department of Mathematics, Central Michigan University, Mount Pleasant, MI, USA}
\email{jordan.watts@cmich.edu}

\date{\today}
\maketitle

\begin{abstract}
A smooth manifold hosts different types of submanifolds, including
embedded, weakly-embedded, and immersed submanifolds. 
The notion of an immersed submanifold 
requires additional structure (namely, the choice of a topology);
when this additional structure is unique, we call the subset
a \emph{uniquely immersed submanifold}.
Diffeology provides yet another intrinsic notion of submanifold:
a \emph{diffeological submanifold}.

We show that from a categorical perspective
diffeology rises above the others:
viewing manifolds as a concrete category over the category of sets,
the \emph{initial morphisms} 
are exactly the (diffeological) \emph{inductions},
which are the diffeomorphisms with diffeological submanifolds.
Moreover, if we view manifolds as a concrete category over the category of 
topological spaces, we recover Joris and Preissmann's notion 
of \emph{pseudo-immersions}.

We show that these notions are all different.
In particular, a theorem of Joris from 1982 yields a diffeological submanifold 
whose inclusion is not an immersion,
answering a question that was posed by Iglesias-Zemmour.
We also characterize local inductions as those pseudo-immersions
that are locally injective.

In appendices, we review a proof of Joris' theorem, 
pointing at a flaw in one of the several other proofs 
that occur in the literature,
and we illustrate how submanifolds inherit paracompactness
from their ambient manifold.
\end{abstract}

\section{Overview}
\labelsec{sec:overview}

This is a mostly-expository paper about intrinsic notions of ``submanifold''.

Weakly-embedded submanifolds --- for example, 
an irrational line in the torus --- are abundant; see Example~\ref{ex:leaves}.
Until some time ago, the senior author --- Yael Karshon --- 
was under the (incorrect) impression 
that the notion of a weakly-embedded submanifold
``obviously'' coincides with the diffeological notion of a submanifold
(when the ambient space is a manifold).
Jordan Watts pointed out that this ``fact'' is not obvious,
and it is equivalent to an open question that Iglesias-Zemmour
posed in his book \cite{Iglesias-Zemmour2013}.
After playing around with these concepts
we came up with the cusp $\{ (x,y) \in \R^2 \ | \ x^2 = y^3 \}$,
but we couldn't determine if it is a diffeological submanifold.
David Miyamoto then discovered Joris' theorem~\cite{Joris1982}, 
which implies that the cusp is a diffeological submanifold,
resolving Iglesias-Zemmour's open question.
The one-page note that we planned to write about this observation
evolved into the current paper.

In Section~\ref{sec:submanifolds}, we compare (embedded)
submanifolds, weakly-embedded submanifolds,
diffeological submanifolds, and uniquely immersed submanifolds.
All of these notions are intrinsic: 
they are properties of a subset $S$ of a manifold that 
do not require an {\it a priori} choice of an additional structure on $S$.
This is in contrast with the notion of ``immersed submanifold'',
which is not intrinsic.
As far as we know, the fourth of these notions --- that of a 
uniquely immersed submanifold --- is new.  
With the help of Joris' theorem,
we show that these four notions of submanifolds are all different.

Section~\ref{sec:maps} is about (diffeological) inductions 
and pseudo-immersions of manifolds.
Iglesias-Zemmour's open question was whether every induction between manifolds
is an immersion; Joris' theorem provides a negative answer.
The related notion of a pseudo-immersion was defined
by Joris and Preissmann~\cite{Joris_Preissmann1987}, 
following a suggestion of Fr\"olicher.
Every induction between manifolds is a pseudo-immersion, 
but being a pseudo-immersion is a local property,
so a pseudo-immersion is not necessarily an induction.
Moreover, a pseudo-immersion need not even be a local induction:
Joris and Preissmann found a pseudo-immersion that is not locally injective.
This naturally raises the question of whether 
every \emph{locally injective} pseudo-immersion is a local induction.
We give a positive answer to this question, in Corollary~\ref{cor:1}.
Finally, in Remark~\ref{rk:categorical}
we show that both inductions and pseudo-immersions 
arise naturally from a categorical point of view, as initial morphisms,
when we view the category of manifolds
as a concrete category over the category of sets,
and, respectively,
as a concrete category over the category of topological spaces.

In Appendix~\ref{sec:joris} we present a proof of Joris' theorem
that is based on that of Amemiya and Masuda~\cite{Amemiya_Masuda1989}.
We also point a subtle error in another author's
shorter alleged proof of this theorem.
In Appendix~\ref{app:diffeology} we review the basics of diffeology.
In Appendix~\ref{app:criterion-smooth-map} we prove
a characterization of inductions
that we used in Remark~\ref{rk:categorical}.
In Appendix~\ref{app:paracompactness} we discuss 
the paracompactness and second countability assumptions
that appear in definitions of a manifold.
Paracompactness has the advantage that this property 
is automatically inherited by diffeological (in particular, weakly-embedded) 
submanifolds.
In Appendix~\ref{app:literature} we compare the concepts of this paper
with similar concepts in some books.

\subsection*{Acknowledgement}
Yael Karshon is grateful for helpful conversations
with Patrick Iglesias-Zemmour and with Asaf Shachar.
This work was partly funded by the Natural Sciences
and Engineering Research Council of Canada.

\section{Submanifolds}
\labelsec{sec:submanifolds}

\subsection*{Submanifolds, diffeological submanifolds, and 
weakly-embedded submanifolds}\ 

A \emph{smooth structure} on a topological space 
is a maximal smooth atlas.
Except where we say otherwise, a \define{manifold} is a set
equipped with a topology that is Hausdorff and paracompact
and with a smooth structure.
\begin{itemize}
\item
A subset $S$ of a manifold $N$ is a \define{submanifold}
(equivalently, \define{embedded submanifold})
if it has a (necessarily unique) manifold structure 
such that a real-valued function $f \colon S \to \R$ 
is smooth if and only if 
each point of $S$ has an open neighbourhood $V$ in~$N$ 
and a smooth function $F \colon V \to \R$
that coincides with $f$ on $S \cap V$.
\end{itemize}

For equivalent definitions, see Remark~\ref{immersed}.

Some authors also refer to \define{properly embedded} submanifolds,
namely, submanifolds whose inclusion map is proper. 
This holds if and only if the submanifold is closed
in the topology of the ambient manifold.

\begin{example}
Open subsets of manifolds are submanifolds.
Regular level sets of smooth functions are properly embedded submanifolds.
\eoe
\end{example}

A \define{diffeological space} is a set $X$, equipped 
with a collection of maps (which are declared to be smooth)
from open subsets of Cartesian spaces to~$X$, that satisfies three axioms.
Manifolds can be identified with those diffeological spaces
that satisfy certain conditions.
See Appendix~\ref{app:diffeology} for details.
The diffeological notion of ``submanifold'', in the special case 
that the ambient diffeological space is a manifold, is equivalent 
to the first of the following two notions.

\bigskip
Let $N$ be a manifold, $S$ a subset of $N$, and $\iota \colon S \to N$ 
the inclusion map.

\begin{itemize}
\item
The subset $S$ is a \define{diffeological submanifold} of $N$ if it has a
(necessarily unique) manifold structure such that the following holds:
\begin{equation}  \labell{condition}
\begin{minipage}{5.5in}
\begin{center}
For any open subset $U$ of any Cartesian space 
and any map $q \colon U \to S$, \\
$q \colon U \to S$ is smooth 
if and only if $\iota \circ q \colon U \to N$ is smooth.
\end{center}
\end{minipage}
\end{equation}
\item
The subset $S$ is a \define{weakly-embedded submanifold} of $N$ if it
is a diffeological submanifold and 
the inclusion map $\iota \colon S \to N$ is an immersion.
\end{itemize}

Every submanifold is a weakly-embedded submanifold, but not every 
weakly-embedded submanifold is a submanifold;\footnote
{
This is an example of the \emph{red herring principle} 
\cite[Chapter~1, footnote on p.~22]{hirsch}: in mathematics, 
a red herring does not have to be either red nor a herring
}
take, for example, the irrational line $\{[t,\sqrt{2}t]\}$ 
in the torus $\R^2/\Z^2$,
or the topologist's sine curve $\{(x,y) \ | \ 
x=0 \text{ or } y=\sin(1/x) \}$ in $\R^2$.
A weakly-embedded submanifold is a submanifold if and only if 
its manifold topology agrees with its subset topology.

Is every diffeological submanifold weakly-embedded? No:
 
\begin{example} \labell{ex:cusp}
The cusp $\{(x,y) \in \RR^2 \mid x^2=y^3\}$ is a diffeological submanifold of $\RR^2$ that is not weakly-embedded; see Figure~\ref{f:cusp}.
	\begin{figure}
		\begin{center} \includegraphics[width=200px]{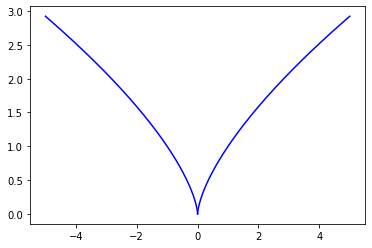} \end{center} 
		\caption{\small The cusp $x^2=y^3$.}\label{f:cusp}
	\end{figure}
\eoe
\end{example}

This follows from a theorem of Henri Joris from 1982 
\cite[Th\'eor\`eme~1]{Joris1982}:

\begin{theorem}[Joris] \labell{thm:1} 
Let $(m, n)$ be a relatively prime pair of positive integers 
and $g \colon U \to \RR$ a real-valued function on a manifold $U$.
Then $g$ is smooth if and only if $g^m$ and $g^n$ are smooth.
\end{theorem}

\begin{proof}[Details for Example~\ref{ex:cusp}]
Equip the cusp with the manifold structure that is given by the
parametrization $t \mapsto (t^3,t^2)$ for $t\in\RR$. 
Joris' theorem implies that
this manifold structure satisfies Condition~\eqref{condition}. 
With this manifold structure, the inclusion map of the cusp into $\R^2$
is not an immersion.
\end{proof}

Weakly-embedded submanifolds are particularly useful:

\begin{example} \labell{ex:leaves}
On a (Hausdorff and paracompact) manifold $N$,
the leaves of any (regular) foliation are weakly-embedded submanifolds 
\cite[Theorem~19.17]{JohnLee}.
So are the leaves of any singular foliation 
\cite{sussman,stefan}.\footnote
{
Under some definitions of ``foliation'' or ``singular foliation''
this is a tautology. Here is a statement  that is not a tautology. 
A \emph{distribution} on a manifold $N$ is a subbundle of $TN$; a 
\emph{singular distribution}
is a subset of $TN$ that near each point is the span of a set of (smooth)
vector fields.
An \emph{integral submanifold} 
is a connected immersed submanifold (see Remark~\ref{immersed}) 
whose tangent space coincides 
with the singular distribution at each of its points. 
A singular distribution is \emph{integrable} if each point is contained 
in an integral submanifold. The \emph{leaves} --- defined as the 
equivalence classes of the relation generated by the integral submanifolds --- 
are then weakly-embedded integral submanifolds.
}
So are the orbits of any Lie group action on $N$ \cite{castrigiano-hayes}
and of any Lie algebroid action on $N$.
In Lie theory, the bijection between Lie subalgebras
and connected Lie subgroups is with ``Lie subgroup" interpreted 
as weakly-embedded \cite[Theorems~19.25 and~19.26]{JohnLee}.
\eoe
\end{example}

\begin{remark}  \labell{rk:criterion}
A subset $S$ of a manifold $N$ is a weakly-embedded submanifold
if and only if the following holds; see Remark~\ref{rk:def weakly embedded}.
\begin{equation} \labell{criterion}
\begin{minipage}{5in}
About each point of $S$ there is a chart 
$\varphi \colon \calO \to \Omega \subset \R^n$ of $N$ \\
that takes the smooth path component of the point in $S \cap \calO$ \\
onto the intersection of $\Omega$ with a linear subspace of $\R^n$.
\end{minipage}
\end{equation}
Here, the \emph{smooth path components} of a subset $S$ of a manifold $N$
are the equivalence classes for the equivalence relation on $S$
that is generated by the smooth paths in $N$ that are contained in $S$.
For example, let $N = \R^2$ and let $S$ be the union of the $y$-axis
with the graph of the function $x\sin(1/x)$ over the positive $x$-axis.
Then $S$ has one path component but two smooth path components
(because the curve $(x,x\sin(1/x))$, for $0 \leq x \leq 1$, 
is not rectifiable).
\eor
\end{remark}

\subsection*{Uniquely immersed submanifolds} \ 

When claiming that a subset $S$ of a manifold $N$ is a manifold, what
does ``is'' mean? This can usually be interpreted as the existence of a
unique manifold structure on $N$ such that some condition holds. 
Being a diffeological submanifold is a reasonable condition; 
being an embedded or weakly-embedded submanifold are special cases. 
In contrast to being an immersed submanifold 
(see Remark~\ref{immersed}),
these properties are intrinsic,
in that they are properties of a subset $S$ of a manifold $N$, 
not additional structures on~$S$.
Here is yet another intrinsic property, which we have not seen
elsewhere in the literature:

\begin{itemize}
\item
A subset $S$ of a manifold $N$ is a \define{uniquely immersed submanifold}
if it has a unique manifold structure such that 
the inclusion map $S \hookrightarrow N$ is an immersion.
\end{itemize}
In contrast to embedded submanifolds, weakly-embedded submanifolds,
and diffeological submanifolds, here ``unique'' is being assumed,
not concluded.

Every weakly-embedded submanifold is uniquely immersed.  
The figure eight in $\R^2$ (cf.\ Remark \ref{immersed}) 
is not uniquely immersed, because it has two distinct manifold
structures with which the inclusion map is an immersion. The cusp of
Example~\ref{ex:cusp} is also not uniquely immersed, because
it does not have any manifold structure with which the inclusion map is
an immersion. In $\R^2$, the union of the $x$-axis and the open positive
$y$-axis is a uniquely immersed submanifold that is not a diffeological
submanifold (hence is not weakly-embedded); 
see Figure~\ref{f:uniquely-immersed}.
	\begin{figure}
		\begin{center} \includegraphics[width=200px]{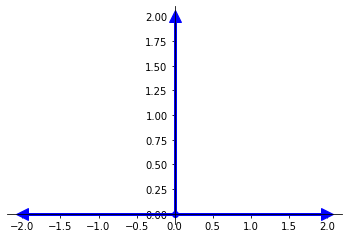} \end{center} 
		\caption{\small The union of the $x$-axis and the open positive $y$-axis is a uniquely immersed submanifold, but is not a diffeological submanifold.}\label{f:uniquely-immersed}
	\end{figure}

In Remark~\ref{immersed} below we also discuss the (non-intrinsic!) notion
of an ``immersed submanifold''.
Every uniquely immersed submanifold, with its induced manifold structure,
is an immersed submanifold.

\section{Initial maps}
\labelsec{sec:maps}

\subsection*{(Diffeological) inductions} \ 

In diffeology, a map $f \colon X \to Y$ is an \define{induction} if it is
a diffeomorphism of $X$ with a subset of~$Y$, as diffeological spaces. In
the special case that the spaces are manifolds, this diffeological notion
becomes equivalent to the following notion.

\begin{itemize}
\item
A map $f \colon M \to N$ between manifolds is an \define{induction} if it is a diffeomorphism of $M$ with a diffeological submanifold of $N$.
\end{itemize}

An immersion is not necessarily an induction, even if it is injective; 
for example, take a parametrization of the figure eight 
(cf.\ Remark \ref{immersed}).
Instead, if $f$ is an immersion, then $f$ is a \define{local induction}: 
each point $x \in M$ has an open neighbourhood $\calO$ in $M$ 
such that $f|_\calO \colon \calO \to N$ is an induction 
(\cite{Iglesias-Zemmour2013}, p.~58, Note). 
In his 2013 book, Iglesias-Zemmour writes ``I still do not know 
if there exist inductions [to] $\R^n$ that are not immersions'' 
(\cite{Iglesias-Zemmour2013}, p.17, note after Exercise~14), 
and ``Actually, it could be weird if local inductions 
between real domains\footnote{
  Iglesias--Zemmour calls open subsets of Cartesian spaces ``real domains''
}
were not immersions'' (\cite{Iglesias-Zemmour2013}, p.58, Note). 
Reformulating Example~\ref{ex:cusp}, 
we resolve Iglesias-Zemmour's open question:

\begin{example} \labell{ex:1}
The smooth map $t \mapsto (t^3,t^2)$ is an induction (by Joris' Theorem 
\ref{thm:1}), but it is not an immersion. 
\eoe
\end{example}

Iglesias-Zemmour writes (\cite{Iglesias-Zemmour2013},
p.58, footnote): ``When J.\ M.\ Souriau wrote one of his first papers on
diffeology, he named immersion what is called induction now, but after a
remark from J.\ Pradines he changed his mind''. 
In an email to David Miyamoto on October~6, 2020, 
Iglesias Zemmour clarified that
Pradines's remark was that the figure eight is an immersion that is not
an induction; Souriau's change to ``induction'' was motivated by this
example. Pradines made his observation while refereeing Iglesias-Zemmour's
thesis, in which Iglesias-Zemmour introduced ``local inductions'',
adopting Souriau's change in nomenclature.

Every induction is smooth. 
For maps between manifolds, we have the following characterization 
of inductions.

\begin{lemma} \labell{lem:induction smooth}
A smooth map $f \colon M \to N$ between manifolds is an induction 
if and only if the following holds:
\begin{equation} \labell{iff}
\begin{minipage}{5.5in}
\begin{center}
For any manifold $U$ and any map $p \colon U \to M$, \\
$p \colon U \to M$ is smooth 
 if and only if $f \circ p \colon U \to N$ is smooth.
\end{center}
\end{minipage}
\end{equation}
\end{lemma}

We prove Lemma~\ref{lem:induction smooth} 
in Appendix~\ref{app:criterion-smooth-map}, but here is the key step:

\begin{lemma}\labell{lem:one-to-one} 
Let $f \colon M \to N$ be a map between manifolds. 
If \eqref{iff} holds, then $f$ is one-to-one.
\end{lemma}

\begin{proof}
Let $x,y$ be points of $M$ such that $f(x)=f(y)$. Define $p \colon \R
\to M$ by $p(t) = x$ if $t \leq 0$ and $p(t) = y$ if $t > 0$. 
Then $f \circ p$ is constant, hence smooth. By~\eqref{iff}, $p$ is smooth, 
hence continuous.  Because $M$ is Hausdorff, this implies that $x=y$.
\end{proof}

\begin{remark} \labell{rk:domain hausdorff}
In Lemmas~\ref{lem:induction smooth} and~\ref{lem:one-to-one},
we can take $M$ and $N$ to be general diffeological spaces,
as long as the D-topology of $M$ is Hausdorff.
If the domain $M$ is not Hausdorff, this does not work:
take, for example, $M := \{0,1\}$ equipped with the coarse diffeology
(all maps $U \to M$ are plots) and $N$ a singleton.
\eor
\end{remark}

A \emph{weak embedding} is a diffeomorphism with a weakly-embedded submanifold.
Weak embeddings are exactly those maps 
that are both an induction and an immersion.
For example, $t \mapsto [t,\sqrt{2}t]$ defines a weak embedding
from $\R$ to $ \R^2/\Z^2$.

\subsection*{Pseudo-immersions} \ 

Pseudo-immersions are similar to inductions, 
except that Condition \eqref{iff} is only required to hold
with test-maps $p \colon U \to M$ that are continuous:
\begin{itemize}
\item
A smooth map $f \colon M \to N$ between manifolds is a
\define{pseudo-immersion} if and only if the following holds:
\begin{equation} \labell{iff v2}
\begin{minipage}{5.5in}
\begin{center}
For any manifold $U$ and any continuous map ${p \colon U \to M}$,  \\
$p \colon U \to M$ is smooth 
if and only if $f \circ p \colon U \to N$ is smooth.
\end{center}
\end{minipage}
\end{equation}
\end{itemize}
The restriction to continuous test-maps makes the following statements true.
(The analogous statements for inductions are false.)
\begin{itemize}
\item[(i)]
Being a pseudo-immersion is a local property:
given a map $f \colon M \to N$ between manifolds, 
if each point $x \in M$ has an open neighbourhood $\calO$ in $M$
such that $f|_\calO \colon \calO \to N$ is a pseudo-immersion,
then $f$ is a pseudo-immersion. And,
\item[(ii)]
Every immersion is a pseudo-immersion.
\end{itemize}
Thus, the notion of a pseudo-immersion does not give rise
to a notion of ``submanifold'' whose structure is entirely induced
from the ambient manifold. 

Every induction is a pseudo-immersion, but not every pseudo-immersion
is an induction.  For example, take the figure eight.

Every local induction is a pseudo-immersion,
but not every pseudo-immersion is a local induction: the map 
 \begin{equation*}
   h(x,y) :=
   \begin{cases}
     (x^2, x^3-xe^{-\frac{1}{|y|}}, y) &\text{if } y \neq 0 \\
     (x^2, x^3, 0) &\text{if } y = 0 ,
   \end{cases}
 \end{equation*}
provided by Joris and Preissmann in \cite{Joris_Preissmann1990}, 
is a pseudo-immersion, but it is not injective in any neighbourhood 
of $(0,0)$, because 
$h(e^{-\frac{1}{2|t|}}, t) = ( e^{-\frac{1}{|t|}} , 0 , t) =
 h(-e^{-\frac{1}{2|t|}},t)$.
This raises the question:
\begin{equation}
\labell{q}
\text{If a pseudo-immersion is locally injective, is it a local induction?}
\end{equation}

We answer this question in the affirmative. 
Recall that a map from a topological space $A$ to a topological space $B$ 
is a \define{topological embedding} if it is a homeomorphism of $A$
with a subset of $B$, taken with its subspace topology.

\begin{exercise} \labell{local homeo}
If a map $f \colon A \to B$ from a locally compact space $A$ to a
Hausdorff space $B$ is continuous and locally injective, 
then each point of $A$ has a neighbourhood $\calO$ 
such that $f|_{\calO}$ is a topological embedding.
\end{exercise}

\begin{remark} \labell{rk:homeo induction}
If a smooth map $f \colon M \to N$ is a pseudo-immersion and a topological embedding, then $f$ is an induction. Indeed, let $p \colon U \to M$ be a map such that $f \circ p \colon U \to N$ is smooth. Because $f$ is a topological embedding, $p$ is continuous. Because $f$ is a pseudo-immersion, $p$ is smooth.
\eor
\end{remark}

Combining Exercise~\ref{local homeo} and Remark~\ref{rk:homeo induction}, 
we obtain the answer to the Question~\eqref{q}:

\begin{corollary}\labell{cor:1}
A smooth map between manifolds is a local induction
if and only if it is a locally injective pseudo-immersion.
\end{corollary}

We summarize the various inclusion relations in Figure~\ref{fig:1}.

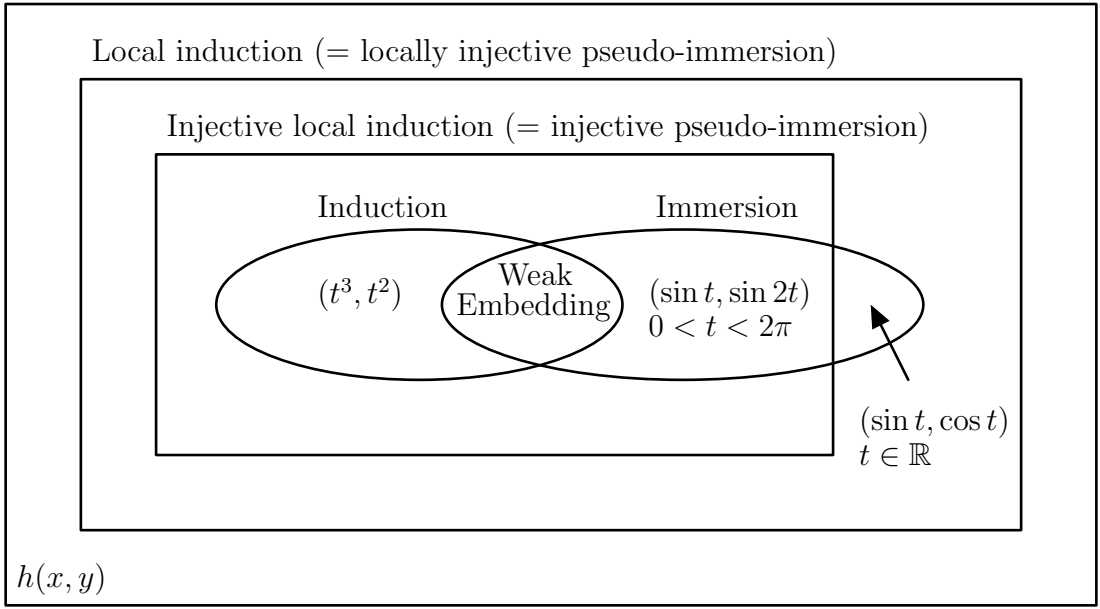
\begin{figure}[h]
  \centering
\begin{tikzpicture}[line cap=round,line join=round,>=triangle 45,x=1cm,y=1cm]
\draw [rotate around={0:(-1.5,0)},line width=1pt] (-1.5,0) ellipse (2.7cm and 1cm);
\draw [rotate around={0:(2,0)},line width=1pt] (2,0) ellipse (3.2cm and 1cm);
\draw (-7,4.7) node[anchor=north west] {Pseudo-immersion};
\draw (-6,3.7) node[anchor=north west] {Local induction (= locally injective pseudo-immersion)};
\draw (-5,2.7) node[anchor=north west] {Injective local induction (= injective pseudo-immersion)};
\draw (-3,1.6) node[anchor=north west] {Induction};
\draw (-0.6, 0.7) node[anchor=north west] {Weak};
\draw (-1.15, 0.3) node[anchor=north west] {Embedding};
\draw (1.5,1.6) node[anchor=north west] {Immersion};
\draw (-3,0.5) node[anchor=north west] {$(t^3, t^2)$};
\draw (1.4,0.5) node[anchor=north west] {$(\sin t, \sin 2t)$};
\draw (1.4,0) node[anchor=north west] {$0<t<2\pi$};
\draw (4.2,-1.2) node[anchor=north west] {$(\sin t, \cos t)$};
\draw (4.2,-1.7) node[anchor=north west] {$t \in \RR$};
\draw (-7,-3.3) node[anchor=north west] {$h(x,y)$};
\draw [line width=1pt] (-7,4) rectangle (7.5,-4);
\draw [line width=1pt] (-6,3) rectangle (6.5,-3);
\draw [line width=1pt] (-5,2) rectangle (4,-2);
\draw [->,line width=1pt] (5,-1) -- (4.5,0);
\end{tikzpicture}
\caption{The inclusion relations for pseudo-immersions, inductions, local inductions, immersions, and their injective and locally injective incarnations. The equalities follow from Corollary \ref{cor:1}.}
\label{fig:1}
\end{figure}

\subsection*{A categorical perspective} 
\begin{remark} \labell{rk:categorical}
Taking a categorical approach to smooth manifolds, 
the natural notion of ``submanifold'' is diffeological, not weakly-embedded.
We now elaborate, following Definition 8.6 of ``The Joy of Cats'' 
\cite{Adamek_Herrlich_Strecker2004}.

Let $\bfA$ be a concrete category over a category $\bfX$;
this means that the category $\bfA$ is equipped 
with a faithful functor $\bfA \to \bfX$.
We write this functor as $A \mapsto |A|$ on objects,
and we consider the set of $\bfA$-morphisms $A \to B$
as a subset of the set of $\bfX$-morphisms $|A| \to |B|$.
An $\bfA$-morphism $f \colon A \to B$ is \textbf{initial}
if the following holds:
\begin{quotation}
\begin{minipage}{5.5in}
\begin{center}
For any $\bfA$-object $C$ and $\bfX$-morphism $g \colon |C| \to |A|$,  \\
$g$ is an $\bfA$-morphism 
  if and only if $f \circ g \colon |C| \to |B|$ is an $\bfA$-morphism.
\end{center}
\end{minipage}
\end{quotation}

Viewing the category of manifolds and smooth maps
as a concrete category over the category of sets and set-maps,
a smooth map $f \colon M \to N$ is an initial morphism
exactly if \eqref{iff} holds.
By Lemma~\ref{lem:induction smooth}, 
the initial morphisms are exactly the inductions.

If, instead, we view the category of manifolds and smooth maps
as a concrete category over the category of topological spaces
and continuous maps, 
a smooth map $f \colon M \to N$ is an initial morphism
if and only if \eqref{iff v2} holds.
Thus, in this context, the initial morphisms 
are exactly the pseudo-immersions.
\eor
\end{remark}

\appendix

\section{Proofs of Joris' Theorem}
\labelsec{sec:joris}

In this appendix we present a proof of Joris' theorem
that is based on that of Amemiya and Masuda~\cite{Amemiya_Masuda1989}.
Here is the strategy:

\begin{enumeratea}
\item Using Boman's theorem\footnote{
A real valued function $g \colon U \to \RR$ on an open subset $U$ 
of a Cartesian space is smooth if the composition $g \circ \gamma$ 
is smooth for any smooth curve $\gamma \colon \RR \to U$.
} 
\cite{Boman1967}, reduce Joris' theorem to a statement about functions 
of a single variable: 
\begin{equation} \labell{hypothesis}
\begin{minipage}{4.5in}
Let $m$ and $n$ be relatively prime positive integers,
and let $g$ be a real-valued function of a single variable
such that $g^m$ and $g^n$ are smooth.
Then $g$ is smooth.
\end{minipage}
\end{equation}

\item Use the continuity of the $m$th or $n$th root to establish that $g$
is continuous, and smooth outside its set of zeros.

\item Prove that $g$ is smooth outside the set $P$ 
of flat zeros of $g^m$.
(The \emph{flat zeros} are the points where all the derivatives vanish.)

\item Prove that $g$ is smooth everywhere.

\end{enumeratea}

We leave (a) and (b) to the reader as an exercise.
Here is a proof of (c):

\begin{proof} [Proof of (c)]
Let $t$ be a non-flat zero of $g^m$. Then $t$ is also a non-flat zero of $g^n$. (Indeed, at each point, either all or none of $g^m$, $g^{mn}$, and $g^n$ are flat.)
By Hadamard's lemma, we can write
\begin{equation*}
g^m(\tau) \, = \, (\tau-t)^M h_1(\tau) \quad \text{and} \quad 
g^n(\tau) \, = \, (\tau-t)^N h_2(\tau)
\end{equation*}
where $M$ and $N$ are positive integers
and where $h_1$ and $h_2$ are smooth functions that are non-vanishing at $t$.
Since $(g^m)^n = (g^n)^m$,
$$ (\tau-t)^{Mn} \, h_1^n(\tau) \, = \, (\tau-t)^{Nm} \, h_2^m(\tau);$$
since $h_1$ and $h_2$ are non-zero at $t$, this implies that $Mn = Nm$.
So $m$ divides $Mn$.  But $\gcd(m,n)=1$, so $m$ must also divide $M$.
Writing $g(\tau) = (g^m(\tau))^{1/m} = (\tau-t)^{M/m} h_1^{1/m}(\tau)$, 
we see that $g$ is smooth near $t$.
\end{proof}

The proof of (d) is harder.
Joris \cite{Joris1982} proved by induction that,
for each positive integer $\sigma$,
\begin{equation*}
  g^{(\sigma)} \text{ exists everywhere,} \quad g^{(\sigma)} \text{ vanishes on }P, \quad \text{and } g^{(\sigma)} \text{ is flat on } P.
\end{equation*}
Joris proceeded by contradiction, and invoked a combinatorial lemma that expressed $(g^m)^{(\sigma-1)}$ in terms of $(g^a)^{\sigma-1}(g^b)^\sigma w_{a,b}$, where $a+b \leq m$ and $w_{a,b}$ is continuous and flat on $P$. This argument was quite involved.
Seven years later, Amemiya and Masuda \cite{Amemiya_Masuda1989} presented a simpler ring-theoretic proof. 
Here is their key lemma, and the proof for (d).

\begin{lemma}[\cite{Amemiya_Masuda1989}] 
\labell{lem:4}
Let $S$ be a subring of a ring $R$, such that if $a \in R$ and $a^r \in S$
for every sufficiently large $r$, then $a \in S$. Then the ring of power
series $S[[x]]$ in $R[[x]]$ has the same property.
\end{lemma}

\begin{proof}[Proof of (d), assuming Lemma~\ref{lem:4}] \ \\
Let $g \colon \R \to \R$ be a function satisfying
the hypothesis of~\eqref{hypothesis}. 
By (b) and (c), $g$ is continuous and is smooth outside the set $P$.
Set
\begin{align*}
 R := \{h \colon \RR \smallsetminus P \to \RR \mid \ & h \text{ is smooth } \}, 
    \\
 S := \{h \colon \RR \smallsetminus P \to \RR \mid \ & h \text{ is smooth, } \\
   & \text{ and $h$ extends to a continuous map $\RR \to \RR$ 
that vanishes on $P$ } \}
\end{align*}
These satisfy the conditions of Lemma \ref{lem:4}. 
Define a homomorphism $J \colon R \to R[[x]]$ by 
$$ J(h) := \sum_{i=0}^\infty x^i \frac{h^{(i)}}{i!}.$$ 

Note that $h:= g|_{\R \smallsetminus P}$ is in $R$.

Because $(m,n)$ is relatively prime, 
every sufficiently large integer $r$
can be written as non-negative integer combination of $m$ and $n$.
Indeed, after possibly switching $m$ and $n$,
there exist integers $a < 0 < b$ such that ${am+bn=1}$.
For every integer $r \geq (-a)mn$,
write $r = (-a)mn + An + j$ with $A \geq 0$ and $0 \leq j < n$;
then $r = (-a)(n-j)m +(A+bj)n$.

For all such $r$, since $g^m$ and $g^n$ are smooth and are flat at $P$, 
so is $g^r$.
So $(h^r)^{(i)} = (g^r)^{(i)}|_{\R \smallsetminus P}$ is in $S$ for all~$i$.
Therefore, $J(h^r)$ is in $S[[x]]$.
Since $J$ is a homomorphism, $J(h)^r = J(h^r)$, so $J(h)^r \in S[[x]]$.  
By Lemma \ref{lem:4}, $J(h) \in S[[x]]$. 
This means that each $h^{(i)}$ is in $S$, 
or in other words that each derivative $g^{(i)}:\R \smallsetminus P \to \R$ 
extends to a continuous map $\R \to \R$ that vanishes on $P$. 
Combined with the fact $g$ is continuous, we may conclude that $g$ is smooth 
(Lemma 2 in \cite{Amemiya_Masuda1989}, 
a consequence of the Mean Value Theorem).
\end{proof}

Another six years later, Myers \cite{Myers2005} proposed a short,
elementary proof of Joris' theorem, using only Rolle's theorem. 
Myres laid out the general strategy of the proof very clearly, but
unfortunately, his proof that $g$ is smooth at points in $P$ lacks
detail. 
In his induction step, Myers shows 
that if $g$ is of type $C^{k-1}$
and its first $k-1$ derivatives vanish at the points of $P$
then its $k$th derivative exists and vanishes at the points of $P$.
But he does not show that this $k$th derivative is continuous 
at the points of $P$.
In fact, his argument seems to apply to any function whose square alone 
is smooth and flat on $P$; as shown by Glaeser \cite{Glaeser1963}, 
such a function need not be smooth.

\section{Diffeology}
\labelsec{app:diffeology}

A \define{diffeological space} is a set $X$ equipped with a collection
of maps $p \colon U \to X$, called plots, from open subsets of Cartesian
spaces to $X$, that satisfies the following three axioms:
\begin{itemize}
\item Constant maps are plots.
\item The precomposition of a plot with a $C^\infty$ map between open subsets of Cartesian spaces is a plot.
\item If the restrictions of $p \colon U \to X$ to elements of an open cover of $U$ are plots, then $p$ is a plot.
\end{itemize}
The \define{subset diffeology} of a subset $A$ of a diffeological space
$X$ consists of those maps $p \colon U \to A$ whose composition with
the inclusion map are plots of $X$.

A map between diffeological space is \define{smooth} if its precomposition
with each plot is a plot. The map is a \define{diffeomorphism} if it is
a bijection and it and its inverse are smooth.

On a diffeological space $X$, the \define{D-topology} is the final
topology that is induced by the set of plots:
a subset of $X$ is open iff its preimage 
under each plot is open in the domain of the plot.

Equipping each manifold $M$ with the set of $C^\infty$ maps 
from open subsets of Cartesian spaces to $M$,
we obtain a fully faithful functor from the category of manifolds
to the category of diffeological spaces. This identifies manifolds with those diffeological spaces
that are locally diffeomorphic to Cartesian spaces
and whose D-topology is Hausdorff and paracompact.
%

\section{Criterion for diffeological induction}
\labelsec{app:criterion-smooth-map}

In this appendix we prove Lemma~\ref{lem:induction smooth},
which gives the necessary and sufficient criterion \eqref{iff}
for a smooth map
$f \colon M \to N$ between manifolds to be a diffeological induction. 

In fact, 
the ``only if'' direction of \eqref{iff} is automatically true
for all smooth $f$ (and is equivalent to the smoothness of $f$),
and the criterion~\eqref{iff} 
is necessary and sufficient for $f \colon M \to N$ being an induction
even if $f$ is just a set-map.

\begin{proof}[Proof of Lemma~\ref{lem:induction smooth}]
Let $f \colon M \to N$ be a map between manifolds,
let $S := \image f$ be its image,
let $\iota \colon S \to N$ be the inclusion map, 
and let $\hat{f} \colon M \to S$ denote the map $f$ viewed as a map to $S$.

First, suppose that~\eqref{iff} holds.
By Lemma~\ref{lem:one-to-one}, $f$ is one-to-one. 
Equip $S$ with the manifold structure with which 
the bijection $\hat{f} \colon M \to S$ is a diffeomorphism. 
Then for any open subset $U$ of any Cartesian space
and any map $q \colon U \to S$,
\begin{equation*}
q \colon U \to S \text{ is smooth \ if and only if \ } p := \hat{f}^{-1} \circ q \colon U \to M \text{ is smooth. }   
\end{equation*}
By~\eqref{iff},
\begin{equation*}
  p \colon U \to M \text{ is smooth \ if and only if \ } f \circ p \colon U \to N \text{ is smooth. } 
\end{equation*}
Because $f \circ p = \iota \circ q$, this implies 
the condition~\eqref{condition} in the definition of diffeological
submanifold.  So $S$ is a diffeological submanifold of $N$. 
Because $\hat{f} \colon M \to S$ is a diffeomorphism, $f$ is an induction.

Next, suppose that $f$ is an induction. 
Then $S$ is a diffeological submanifold of $N$, 
and --- with this manifold structure on $S$ --- 
the map $\hat{f} \colon M \to S$ is a diffeomorphism. 
Because smoothness is a local property,
the property \eqref{condition} of $S$ also holds with $U$ an arbitrary manifold
(and not only an open subset of a Cartesian space).  
Fix a manifold $U$ and a map $p \colon U \to M$.
Because $\hat{f} \colon M \to S$ is a diffeomorphism,
\begin{equation*}
p \colon U \to M \text{ is smooth \ if and only if \ } \hat{f} \circ p \colon U \to S \text{ is smooth. }  
\end{equation*}
By the version of~\eqref{condition} for manifolds $U$,
\begin{equation*}
\hat{f} \circ p \colon U \to S \text{ is smooth \ if and only if \ } \iota \circ (\hat{f} \circ p) \colon U \to N \text{ is smooth. }  
\end{equation*}
Because $\iota \circ \hat{f} \circ p = f \circ p$, we obtain~\eqref{iff}.
\end{proof}

\section{Paracompactness and second countability}
\labelsec{app:paracompactness}

We require manifolds to be Hausdorff and paracompact. Some authors impose
the stronger requirement that manifolds be Hausdorff and second countable.
With our definition, every connected component of a manifold is second
countable, but there can be uncountably many connected components.
See, e.g., \cite[p.~30, Problem 1-5]{JohnLee}.

\begin{remark}[Topological assumptions are superfluous]
\labell{superfluous}
In the definitions of an ``embedded'', ``weakly-embedded'',
or ``diffeological'' submanifold $S$, we require \emph{a priori} 
that its manifold topology be Hausdorff and paracompact.
But since the ambient manifold $N$ is Hausdorff and paracompact
and the inclusion map is continuous, these assumptions are superfluous;
see Theorem~\ref{thm:3}.
Similarly, an embedded submanifold of a second countable manifold
is automatically second countable.
In contrast, a weakly-embedded submanifold of a second countable manifold
need not be second countable.
For example, the set of irrational numbers in $\R$, 
with its discrete zero-dimensional manifold structure,
satisfies~\eqref{condition}, so it is a weakly-embedded submanifold.
\eor
\end{remark}

\begin{theorem} \labell{thm:3} \ 
Any locally Cartesian topological space $S$ 
that admits an injective continuous map 
to a (Hausdorff and paracompact) manifold $M$
is Hausdoff and paracompact.
\end{theorem}

This result is adapted from Chapter 1, Section 11, Part 7, Theorem 1 of
Bourbaki's \emph{General Topology}, \cite{Bourbaki1966}. 
Bourbaki's result is stated with more general assumptions, 
and its proof is longer. 

\begin{proof}[Proof of Theorem \ref{thm:3}]
Let $S$ be a locally Cartesian topological space, and let  
$$ i \colon S \to M $$
be an injective continuous map.

Let $x,y$ be distinct points of $S$. 
Because $i$ is injective, $i(x) \neq i(y)$.
Because $M$ is Hausdorff, $i(x)$ and $i(y)$ have disjoint
neighbourhoods $U$, $V$ in $M$.
Because $i$ is continuous, the preimages in $S$ of $U$ and $V$
are disjoint neighbourhoods of $x$ and $y$ in $S$.
This shows that $S$ is Hausdorff.

Recall that a locally Cartesian Hausdorff topological space
is paracompact if and only if each of its connected components
is second countable.  
(See, e.g., \cite[page~459]{spivak}.)
Without loss of generality, we will now assume that $S$ and $M$ 
are connected, and we will show that $S$ is second countable.

{
Let $\fk{B}$ denote the collection of those subsets of $S$ 
that are homeomorphic to closed balls in Cartesian spaces. 
Each element of $\fk{B}$ is second countable.
We will describe a countable open cover of $S$ 
that refines $\fk{B}$.
Each element of this cover is an open subset of a second countable space,
hence is second countable.
Since the cover is countable, we conclude $S$ is second countable.

Fix a countable basis $\fk{B}_M$ of the topology of $M$.

We call a pair $(W,U)$ \define{distinguished} if 
$U$ is an element of $\fk{B}_M$
and $W$ is a component of $i^{-1}(U)$
that is contained in some element of $\fk{B}$. 
We need two facts about such pairs:

\begin{enumerate}[label = (\alph*)]
\item \emph{For any point $x \in S$, 
there is a distinguished pair $(W, U)$ with $x \in W$}:

Indeed, fix a point $x \in S$, take $V \in \fk{B}$ 
whose interior contains $x$,
and let $F := V \cap \overline{(S \smallsetminus V)}$ 
be the frontier of~$V$.
Then $F$ is compact and it does not contain~$x$.
Because $i$ is one-to-one, $i(F)$ does not contain $i(x)$;
because $i$ is continuous, $i(F)$ is compact;
because $M$ is Hausdorff, $i(F)$ is closed.
So $M \smallsetminus i(F)$ is an open neighbourhood of $i(x)$ in $M$.
Let $U$ be an element of $\fk{B}_M$ that contains $i(x)$ 
and is contained in $M \smallsetminus i(F)$. 
Set $W$ to be the component of $i^{-1}(U)$ that contains $x$. 
Then $W$ is contained in (the interior of) $V$,
and the pair $(W, U)$ is distinguished.

\item \emph{For any distinguished pair $(W, U)$,
there are at most countably many distinguished pairs $(W', U')$ 
in which $W'$ intersects $W$}:

Because $\fk{B}_M$ is countable,
it is enough to show that, for each $U' \in \fk{B}_M$,
the set of components $W'$ of $i^{-1}(U')$ that intersect $W$ is countable.
Because $S$ is locally connected, each such $W'$ is open in $S$.
Because (having fixed $U'$) the sets $W'$ are disjoint,
and because they are open subsets of $W$,
which is separable (it is contained in an element of $\fk{B}$, 
which is homeomorphic to a closed ball), 
there can be at most countably many such sets $W'$.
\end{enumerate}

Now, define $x \sim x'$ if there are distinguished pairs 
$(W_1, U_1), \ldots, (W_n, U_n)$ with $x \in W_1$, $x' \in W_n$, 
and consecutive $W_i$ intersect. 
This is an equivalence relation; (a) gives reflexivity. 
Its equivalence classes are open; because $S$ is connected, 
there is only one equivalence class.

Fix $x$.  By (a) there exists a distinguished pair $(W_1, U_1)$ 
with $x\in W_1$; set $C_1 := W_1$. 
Define $C_n$ recursively to be the union of the $W$ 
from distinguished pairs $(W, U)$ such that $W$ intersects $C_{n-1}$. 
Each $C_n$ is a countable union by (b). 
Furthermore, the collection of $C_n$ cover $S$: a chain $(W_i', U_i')$ 
from $x$ to $x'$ witnesses $W_i' \subseteq C_{i+1}$ for each $i$. 
We take our open cover to be given by those $W$ that occur in this 
construction.
}
\end{proof}

\begin{remark}
In Theorem~\ref{thm:3},
if $S$ is a not-necessarily-paracompact manifold and $i$ is an immersion,
an alternative argument is to equip $M$ with a Riemannian metric, 
pull it back to a Riemannian metric on $S$,
and note that connected Riemannian manifolds are metrizable 
and connected metrizable manifolds are second countable
(see \cite[Theorem 7, page 315]{spivak} and \cite[page~459]{spivak}).
The advantage of Theorem \ref{thm:3} is that it is
purely topological, and it applies 
to diffeological submanifolds that are not necessarily weakly-embedded.
\eor
\end{remark}

\section{Remarks on the literature}
\labelsec{app:literature}

\subsection*{Submanifolds in some textbooks} \ 

Our notions of ``(embedded) submanifold'' and ``weakly-embedded submanifold'' 
are equivalent to those in standard textbooks 
such as John Lee's \cite[Chapter~5]{JohnLee},
except that some authors (including John Lee)
require their manifolds and submanifolds to be second countable 
and not only paracompact; see Appendix~\ref{app:paracompactness}.

\begin{remark} \labell{embedded via slice}
Our definition in \S\ref{sec:submanifolds} for what it means
for a subset $S$ of a manifold $N$ to be an ``(embedded) submanifold''
is essentially an unravelling of the definition of ``submanifold'' 
in Guillemin and Pollack~\cite{guillemin-pollack}.
A different, but equivalent, definition in the literature
is that about each point of $S$ there is a chart
$\varphi \colon \calO \to \Omega \subset \R^n$ of $N$
that takes $S \cap \calO$ onto the intersection of $\Omega$
with a linear subspace of $\R^n$.
(Cf.~Remark~\ref{rk:criterion}.)
\end{remark} 

\begin{remark} \labell{immersed}
An alternative approach is through the notion of an 
\emph{immersed submanifold}, defined as a subset $S$, 
equipped with a manifold structure,
such that the inclusion map is an immersion. 
This notion is not an intrinsic property of the subset $S$:
it requires the choice of a manifold structure on $S$.

An (embedded) submanifold can be defined as an immersed submanifold $S$
whose manifold topology agrees with its subset topology,
and a weakly embedded submanifold can be defined 
as an immersed submanifold $S$ such that Condition~\eqref{condition} holds;
see e.g.\ \cite{JohnLee}.
In either of these cases, the manifold structure on $S$ is unique.
Our definitions of submanifold and weakly-embedded submanifold
are intrinsic; they avoid an {\it a priori} choice of a manifold structure 
on $S$.

Not every immersed submanifold is weakly-embedded. 
For example, the parametrizations of the figure eight in $\R^2$ 
as $(\sin t , \sin (2t))$ for $0 < t  < 2\pi$ 
and for $-\pi < t < \pi$
exhibit it as two distinct immersed submanifolds.
\eor
\end{remark}

\begin{remark} \labell{rk:def weakly embedded}
According to Kolar-Michor-Slovak \cite[Def.~2.14]{KMS1993},
an \emph{initial submanifold} of a manifold $N$
is a subset $S$ of $N$ that satisfies the criterion \eqref{criterion}.
Lemma~2.15 of their book \cite{KMS1993} can be rephrased 
as saying that every weakly-embedded submanifold is an initial submanifold.
Lemmas~2.16 and~2.17 of their book~\cite{KMS1993} can be rephrased 
as saying that every initial submanifold is a weakly-embedded submanifold.
\eor
\end{remark}

\subsection*{Initial maps in some textbooks} \ 

Kolar-Michor-Slovak \cite[2.10]{KMS1993}
refer to smooth maps that satisfy the property \eqref{iff} 
as \emph{initial}.
This term is consistent with \cite{Adamek_Herrlich_Strecker2004};
see Remark~\ref{rk:categorical}.
Kolar-Michor-Slovak do not note that initial maps are automatically injective,
which is the content of our Lemma~\ref{lem:one-to-one}.
By our Lemma~\ref{lem:induction smooth}, 
initial maps coincide with (diffeological) inductions.

While Kolar-Michor-Slovak do not work with diffeology, 
they note \cite[Remark~2.13]{KMS1993} 
that, by Joris' theorem (Theorem~\ref{thm:1}), 
the cusp $t \mapsto (t^m, t^n)$ is initial, but is not an immersion. 
This coincides with our observation in Example \ref{ex:1} that the cusp 
is an induction that is not an immersion.
Kolar-Michor-Slovak use the cusp as evidence that ``to look for all smooth
[initial mappings] is too difficult'' \cite[Remark~2.10]{KMS1993}.
They then focus on initial injective immersions;
these coincide with weak embeddings.
They do not introduce a term for weak embeddings,
but their notion of an ``initial submanifold'',
which they define using charts, is equivalent to the notion 
of a weakly-embedded submanifold; see Remark \ref{rk:def weakly embedded}.

Maps that satisfy the property \eqref{iff} also appear 
in Jeffrey Lee's book \cite[Def.~3.10]{Lee2009};
he calls them \emph{smoothly universal}.
He too does not note that such maps are necessarily injective.
By our Lemma~\ref{lem:induction smooth}, these maps are exactly the inductions.
Jeffrey Lee defines \emph{weak embeddings} to be 
smoothly universal injective immersions \cite[Def.~3.11 on p.~129]{Lee2009};
this agrees with our usage of this term.

Henri Joris and Emmanuel Preissmann introduced the notion of
a pseudo-immersion in their 1987 paper \cite{Joris_Preissmann1987}. 
Henri Joris had studied such maps already in 1982 \cite{Joris1982}, 
at the suggestion of Alfred Fr\"{o}licher;
he called their defining property \eqref{iff v2} a ``universal property''.
According to Joris,
it was Fr\"{o}licher who suggested the example $t \mapsto (t^2,t^3)$
as a candidate for a pseudo-immersion that is not an immersion.



\end{document}